\documentclass{article}%
\usepackage{amsfonts}
\usepackage{amsmath}
\usepackage{amssymb}
\usepackage{graphicx}%
\providecommand{\U}[1]{\protect\rule{.1in}{.1in}}
\numberwithin{equation}{section}
\newtheorem{theorem}{Theorem}[section]
\newtheorem{acknowledgement}[theorem]{Acknowledgement}

\newtheorem{lemma}[theorem]{Lemma}

\newenvironment{proof}[1][Proof]{\noindent\textbf{#1.} }{\ \rule{0.5em}{0.5em}}

\begin{document}

\title{Asymptotic formulae for $s-$numbers of a Sobolev embedding and a Volterra type operator}
\author{David E. Edmunds
\and Jan Lang}
\date{}
\maketitle

\begin{abstract}
Sharp upper and lower estimates are obtained of the approximation numbers of a
Sobolev embedding and an integral operator of Volterra type. These lead to
asymptotic formulae for the approximation numbers and certain other $s-$numbers.



\end{abstract}

\section{Introduction}

Let $\Omega$ be a bounded open subset of $\mathbb{R}^{n}$ ($n\in\mathbb{N})$
with smooth boundary, and suppose that $p\in(1,\infty),$ $m\in\mathbb{N}.$ It
is a familiar fact that the embedding $T$ of the Sobolev space $W_{p}%
^{m}(\Omega)$ in $L_{p}(\Omega)$ is compact and that its approximation numbers
$a_{k}(T)$ decay like $k^{-m/n};$ that is, there are positive constants
$c_{1},c_{2}$ such that for all $k\in\mathbb{N},$%
\[
c_{1}\leq k^{m/n}a_{k}(T)\leq c_{2}.
\]
The same holds for the embedding $E$ of $\overset{0}{W}$\/$_{p}^{m}(\Omega)$
(the closure of $C_{0}^{\infty}(\Omega)$ in $W_{p}^{m}(\Omega))$ in
$L_{p}(\Omega),$ with no restriction on the boundary of $\Omega.$ For these
results, and much more general ones, together with some historical remarks
about their development, see \cite{ET} and \cite{Har}. What is not so clear is
whether or not there is a genuine asymptotic formula for these approximation
numbers; that is, for example, whether or not
\[
\lim_{k\rightarrow\infty}k^{m/n}a_{k}(E)\text{ exists.}%
\]
When $n>1,$ the only case in which an answer is known is when $p=2.$ For
example, if $m=1$ and $p=2,$ then $a_{k}(E)=\lambda_{k}^{-1/2},$ where
$\lambda_{k}$ is the $k^{th}$ eigenvalue (arranged in increasing order and
repeated according to multiplicity) of the Dirichlet Laplacian, and from the
known behaviour of these eigenvalues it follows that
\[
\lim_{k\rightarrow\infty}k^{1/n}a_{k}(E)=\frac{\left\vert \Omega\right\vert
^{1/n}}{2\sqrt{\pi}\left(  \Gamma(1+n/2)\right)  ^{1/n}}.
\]
Similar results hold when $m>1,$ always provided that $p=2.$

When $n=1$ and $\Omega=(a,b),$ even though the technical difficulties
encountered are naturally fewer, the only result when $p\neq2$ of which we are
aware is that of \cite{EL1}. In this it is shown that if $m=1,$ then not only
is there a positive answer to our question, but in fact the approximation
numbers can be calculated precisely, and
\[
a_{k}(E)=\gamma_{p}(b-a)/k\text{ \ }(k\in\mathbb{N}),
\]
where
\[
\gamma_{p}=\frac{1}{2\pi}(p^{\prime})^{1/p}p^{1/p^{\prime}}\sin(\pi/2).
\]
In the present paper, where we study the case $n=1,\Omega=(a,b)$ and $m=2,$ no
such precision is obtained but we do have an asymptotic result: for any
$p\in(1,\infty),$ there is a constant $C=C(p)$ such that for the embedding
$E:\overset{0}{W}$\/$_{p}^{2}(\Omega)\rightarrow L_{p}(\Omega),$
\[
\lim_{k\rightarrow\infty}k^{2}a_{k}(E)=C(b-a)^{2}.
\]
We also obtain a similar asymptotic result for the Volterra type operator
$T:L_{p}((a,b))\rightarrow$ $L_{p}((a,b))$ given by $Tf(x):=\int_{a}%
^{x}(x-t)f(t)dt.$ This can be compared to the result obtained, when $p=2,$ by
Newman and Solomyak for the singular numbers of a map of weighted fractional
integration type, the action taking place in $L_{2}(\mathbb{R}_{+})$ (see
\cite{NS}).

The proof involves the characterisation in variational form of the principal
eigenvalue of the biharmonic operator with Navier boundary conditions, plus
sharp upper estimates for the approximation numbers and sharp lower estimates
for the Bernstein numbers. Since the Bernstein numbers are dominated by the
approximation numbers, the limiting assertion follows. It also holds for
certain other $s-$numbers, such as the Bernstein, Gelfand and Weyl numbers.

\section{Preliminaries}

Let $X$ and $Y$ be Banach spaces with norms $\left\Vert \cdot\right\Vert _{X}$
, $\left\Vert \cdot\right\Vert _{Y}$ respectively, let $B(X,Y)$ be the space
of all bounded linear maps from $X$ to $Y,$ let $T\in B(X,Y)$ and suppose that
$n\in\mathbb{N}.$ The $n^{th}$ approximation number $a_{n}(S)$ of $S$ is given
by%
\[
a_{n}(S)=\inf\left\{  \left\Vert T-F\right\Vert :F\in B(X,Y),\text{ rank
}F<n\right\}  ;
\]
its $n^{th}$ Bernstein number is
\[
b_{n}(S)=\sup\text{ }\inf_{x\in X_{n}\backslash\{0\}}\left\Vert Tx\right\Vert
_{Y}/\left\Vert x\right\Vert _{X},
\]
where the supremum is taken over all $n-$dimensional linear subspaces $X_{n}$
of $X.$ Note that $b_{n}(S)\leq a_{n}(S)$ for all $n$ and all $S.$ In fact,
the approximation numbers are the largest $s-$numbers and the Bernstein
numbers are the smallest injective strict $s-$numbers: for these results and
the terminology used to describe them, see \cite{EL2}, Chapter 5. Here we
simply mention that the class of injective strict $s-$numbers includes the
Gelfand and Weyl numbers, as well as the Bernstein numbers.

Let $p\in(1,\infty),$ let $a,b\in\mathbb{R}$ with $a<b,$ and put
$I=(a,b),\left\vert I\right\vert =b-a.$ The norm on the Lebesgue space
$L_{p}(I)$ will be denoted by $\left\Vert \cdot\right\Vert _{p,I}.$ Given any
$k\in\mathbb{N},$ the Sobolev space $W_{p}^{k}(I)$ is defined to be the set of
all functions $u\in L_{p}(I)$ such that for each $j\in\{1,...,k\}$ the
distributional derivative $u^{(j)}$ \ exists (written as $u^{\prime}%
,u^{\prime\prime}$ when $j=1,2)$ and also belongs to $L_{p}(I);$ $\overset
{0}{W}$\/$_{p}^{k}(I)$ will stand for the closure of $C_{0}^{\infty}(I)$ in
$W_{p}^{k}(I)$ equipped with the norm $\|u\|_{\overset{0}{{W}_{p}^{k}}(I)}:=
\|u^{(k)}\|_{p,I}.$

We recall that if $u\in W_{p}^{k}(I),$ then $u^{(k-1)}$ is absolutely
continuous on $\overline{I\ };$ if $u\in\overset{0}{W}$\/$_{p}^{k}(I)$ \ then
additionally we have $u^{(j)}(a)=u^{(j)}(b)=0$ for $j=0,1,...,k-1.$\ \ 

The papers \cite{BD}, \cite{DO} consider an eigenvalue problem for the
$p-$biharmonic operator with Navier boundary conditions. The one-dimensional
form of this problem is
\begin{equation}
\left(  \left\vert u^{\prime}\right\vert ^{p-2}u^{\prime}\right)  ^{^{\prime}%
}=\lambda\left\vert u\right\vert ^{p-2}u\text{ in }I,\text{ }u=u^{\prime
\prime}=0\text{ at the endpoints }a,b. \label{Eq 2.1}%
\end{equation}
It is shown that ($\ref{Eq 2.1})$ has a least eigenvalue $\lambda,$ which is
positive, simple and isolated, and is given by%
\[
\lambda_{1}=\min\frac{\left\Vert u^{\prime\prime}\right\Vert _{p,I}^{p}%
}{\left\Vert u\right\Vert _{p,I}^{p}},
\]
where the minimum is taken over all $u\in W_{p}^{2}(I)\cap\overset{0}{W}%
$\/$_{p}^{1}(I),$ $u\neq0.$ It is convenient for us to deal with the
reciprocal of this quotient, and we set%
\begin{equation}
J^{0}(I)=J^{0}(a,b)=\sup\frac{\left\Vert u\right\Vert _{p,I}}{\left\Vert
u^{\prime\prime}\right\Vert _{p,I}}, \label{Eq 2.2}%
\end{equation}
where the supremum is taken over all $u\in W_{p}^{2}(I)\cap\overset{0}{W}%
$\/$_{p}^{1}(I),$ $u^{\prime\prime}\neq0,$ so that $u(a)=u(b)=0;$ thus
$J^{0}(I)=\lambda_{1}^{-1/p}.$ Denoting by $f$ the extremal function for
(\ref{Eq 2.2}), we have from \cite{DO} that $f$ is symmetric about the
mid-point $(a+b)/2$ of the interval $I$ and that $f^{\prime}\left(  \left(
a+b\right)  /2\right)  =0.$ Two further quantities related to $J^{0}(I)$ will
be needed: these are \
\begin{equation}
\quad J^{a}(I):=\sup_{u^{\prime\prime}\not =0,u(a)=0}{\frac{\Vert u\Vert
_{p,I}}{\Vert u^{\prime\prime}\Vert_{p,I}}}, \label{Eq 2.3}%
\end{equation}
and
\begin{equation}
J^{b}(I):=\sup_{u^{\prime\prime}\not =0,u(b)=0}{\frac{\Vert u\Vert_{p,I}%
}{\Vert u^{\prime\prime}\Vert_{p,I}}}. \label{Eq 2.4}%
\end{equation}
Here our notation means that the suprema are taken over all those non-zero
$u\in W_{p}^{2}(I)$ that vanish at $a,b$ respectively.

Let $I_{1}=(a,(a+b)/2)$ and $I_{2}=((a+b)/2,b).$ Then from the above
observations we have that $f_{1}:=\chi_{I_{1}}f$ and $f_{2}:=\chi_{I_{2}}f$
are the extremal functions for $J^{a}(I_{1}),J^{b}(I_{2})$ respectively. Also
we have
\begin{equation}
J^{0}(I)=J^{a}(I_{1})=J^{b}(I_{2}), \label{Eq 2.5}%
\end{equation}
and by scaling we obtain
\begin{equation}
J^{0}(I)=|I|^{2}J^{0}((0,1)). \label{Eq 2.6}%
\end{equation}
Next, we let
\begin{equation}
\mathcal{A}^{-}(I):=\sup_{u^{\prime\prime}\not =0,u(a)=u^{\prime}(a)=0}%
{\frac{\Vert u-u(b)\Vert_{p,I}}{\Vert u^{\prime\prime}\Vert_{p,I}}}
\label{Eq 2.7}%
\end{equation}
and
\begin{equation}
\mathcal{A}^{+}(I):=\sup_{u^{\prime\prime}\not =0,u(b)=u^{\prime}(b)=0}%
{\frac{\Vert u-u(a)\Vert_{p,I}}{\Vert u^{\prime\prime}\Vert_{p,I}},}
\label{Eq 2.8}%
\end{equation}
with the same understanding about the notation as above. As the following
lemma shows, these are not really new quantities.

\begin{lemma}
\label{Lemma 2.1}
\begin{equation}
\mathcal{A}^{+}(I)=J^{a}(I)\ \text{and }\mathcal{A}^{-}(I)=J^{b}(I).
\label{Eq 2.9}%
\end{equation}

\end{lemma}

\begin{proof}
We prove the first equality only as the second follows by a symmetric
argument. Let $f$ be the extremal function for $J^{a}(I)$. Then $f(a)=0$, $f$
is increasing on $I$ and $f^{\prime}(b)=0$. Let $v:=f(b)-f,$ so that $v(b)=0,$
$v^{\prime}(b)=-f^{\prime}(b)=0$ and
\[
J^{a}(I)={\frac{\Vert f\Vert_{p,I}}{\Vert f^{\prime\prime}\Vert_{p,I}}}%
={\frac{\Vert v-v(a)\Vert_{p,I}}{\Vert v^{\prime\prime}\Vert_{p,I}}}\leq
\sup_{u^{\prime\prime}\not =0,u(b)=u^{\prime}(b)=0}{\frac{\Vert u-u(a)\Vert
_{p,I}}{\Vert u^{\prime\prime}\Vert_{p,I}}}=\mathcal{A}^{+}(I).
\]

Now let $u$ be the extremal function for $\mathcal{A}^{+}(I)$. With
$g:=u-u(a)$ we have $g(a)=0,$ $g^{\prime}(b)=0$ and
\[
\mathcal{A}^{+}(I)={\frac{\Vert u-u(a)\Vert_{p,I}}{\Vert u^{\prime\prime}%
\Vert_{p,I}}}={\frac{\Vert g\Vert_{p,I}}{\Vert g^{\prime\prime}\Vert_{p,I}}%
}\leq J^{a}(I),
\]
which concludes the proof.
\end{proof}

\medskip

Now let $J=(c,d)\subset(a,b)=I;$ set
\begin{equation}
K_{J}f(x):=u(c)+{\frac{u(d)-u(c)}{d-c}}(x-c)\text{ \ }(x\in J) \label{Eq 2.10}%
\end{equation}
and (with the same convention about notation as before)
\begin{equation}
\mathcal{B}(J)=\mathcal{B}(c,d):=\sup_{u^{\prime\prime}\not =0}{\frac{\Vert
u-K_{J}u\Vert_{p,J}}{\Vert u^{\prime\prime}\Vert_{p,J}}}. \label{Eq 2.11}%
\end{equation}

Note that $u(c)-K_{J}u(c)=u(d)-K_{J}u(d)=0$.

\begin{lemma}
\label{Lemma 2.2} When $p\in(1,\infty)$ and $I=(a,b)\subset\mathbb{R},$%
\[
\mathcal{B}(I)=J^{0}(I).
\]

\end{lemma}

\begin{proof}
{} Let $u$ be the extremal function for $\mathcal{B}(I)$ and set $g=u-K_{I}u$.
Then $g(a)=g(b)=0$ and
\[
\mathcal{B}(I)={\frac{\Vert u-K_{I}u\Vert_{p,I}}{\Vert u^{\prime\prime}%
\Vert_{p,I}}}={\frac{\Vert g\Vert_{p,I}}{\Vert g^{\prime\prime}\Vert_{p,I}}%
}\leq\sup_{f^{\prime\prime}\not =0,f(a)=f(b)=0}{\frac{\Vert f\Vert_{p,I}%
}{\Vert f^{\prime\prime}\Vert_{p,I}}}=J^{0}(I).
\]

To establish the opposite inequality, denote by $f$ the extremal function for
$J^{0}(I)$. Then $f(a)=f(b)=0$ and $K_{I}f(x)=0$ on $I$. Hence
\[
J^{0}(I)={\frac{\Vert f\Vert_{p,I}}{\Vert f^{\prime\prime}\Vert_{p,I}}}%
={\frac{\Vert f-K_{I}f\Vert_{p,I}}{\Vert f^{\prime\prime}\Vert_{p,I}}}%
\leq\mathcal{B}(I)
\]
and the proof is complete.
\end{proof}

\section{Estimates for $s$-numbers of a Sobolev embedding}

Our concern here is with the embedding $E:\overset{0}{W}$\/$_{p}%
^{2}(I)\rightarrow L_{p}(I):$ we repeat that $1<p<\infty$ and $I=(a,b)\subset
\mathbb{R}.$ We begin with an upper estimate for the approximation numbers of
$E.$

\begin{theorem}
\label{Theorem 3.1}For all $n\in\mathbb{N},$
\begin{equation}
a_{n+1}(E)\leq{\frac{|I|^{2}}{n^{2}}}\mathcal{B}(0,1). \label{Eq 3.1}%
\end{equation}

\end{theorem}

\begin{proof}
Let $n\in\mathbb{N}$ and let $I=\cup_{i=0}^{n}I_{i}$, where $I_{i}%
=(a_{i},a_{i+1}]$ for $i=0,1,...,n-1,$ $I_{n}=(a_{n},a_{n+1}),$ with
\begin{align*}
a_{0}  &  =a,a_{1}=a+{\frac{|I|}{2n}},\\
a_{i}  &  =a+{\frac{|I|}{2n}}+(i-1){\frac{|I|}{n}}\quad\text{for }1<i\le n,
\end{align*}
and $a_{n+1}=b.$ Then we have $2|I_{0}|=2|I_{n}|=|I_{i}|=(b-a)/n$ for $0<i<n$.
For any function $f$ on $I$ set
\[
Kf(x):=\chi_{I_{0}}(x)f(a)+\chi_{I_{n}}(x)f(b)+\sum_{i=1}^{n-1}\chi_{I_{i}%
}(x)K_{I_{i}}f(x),
\]
where $K_{I}$ is given by (\ref{Eq 2.10}). With each supremum taken over all
functions $f$ in $W_{0}^{2,p}(I)\backslash\{0\}$ we have%

\begin{align*}
a_{n+1}^{p}(E)  &  \leq\sup{\frac{\Vert f-Kf\Vert_{p,I}^{p}}{\Vert
f^{\prime\prime}\Vert_{p,I}^{p}}}\\
&  =\sup{\frac{\Vert f\Vert_{p,I_{0}}^{p}+\Vert f\Vert_{p,I_{n}}^{p}%
+\sum_{i=1}^{n-1}\Vert f-K_{I_{i}}f\Vert_{p,I_{i}}^{p}}{\sum_{i=0}^{n}\Vert
f^{\prime\prime}\Vert_{p,I}^{p}}}\\
&  \leq\sup{\frac{\left\Vert f^{\prime\prime}\right\Vert _{p,I_{0}}%
^{p}[\mathcal{A}^{+}(I_{0})]^{p}+\left\Vert f^{\prime\prime}\right\Vert
_{p,I_{n}}^{p}[\mathcal{A}^{-}(I_{n})]^{p}+\sum_{i=1}^{n-1}\Vert
f^{\prime\prime}\Vert_{p,I_{i}}^{p}[\mathcal{B}(I_{i})]^{p}}{\sum_{i=0}%
^{n}\Vert f^{\prime\prime}\Vert_{p,I}^{p}}.}%
\end{align*}

Now use of the facts that%
\[
\mathcal{A}^{+}(I_{0})=\mathcal{A}^{-}(I_{n})=\mathcal{B}(I_{i})={{\frac
{|I|^{2}}{n^{2}}}}\mathcal{B}(0,1)\text{ for }0<i<n,
\]
which follow from Lemmas \ref{Lemma 2.1} and \ref{Lemma 2.2}, together with
(\ref{Eq 2.5}), shows that
\[
a_{n+1}^{p}(E)\leq\sup{\frac{\sum_{i=0}^{n}\Vert f^{\prime\prime}\Vert
_{p,I}^{p}(\mathcal{B}(0,1)|I|^{2}/n^{2})^{p}}{\sum_{i=0}^{n}\Vert
f^{\prime\prime}\Vert_{p,I}^{p}}}=\left(  {\frac{\mathcal{B}(0,1)|I|^{2}%
}{n^{2}}}\right)  ^{p}.
\]

\end{proof}

To prove the reverse estimate we need the next elementary technical lemma.

\begin{lemma}
\label{Lemma 3.2} Let $1<p<\infty$ and $f$ be an odd increasing continuous
function on $(-1,1)$. Then
\[
\Vert f\Vert_{p,(-1,1)}=\inf_{\lambda\in\Re}\Vert f-\lambda\Vert_{p,(-1,1)}.
\]
{}
\end{lemma}

\begin{theorem}
\label{Theorem 3.3}For all $n\in\mathbb{N},$
\[
b_{n-1}(E)\geq{\frac{|I|^{2}}{n^{2}}}\mathcal{B}(0,1).
\]

\end{theorem}

\begin{proof}
Let $n\in\mathbb{N}$ and write $I=\cup_{i=1}^{n}J_{i}$, where $J_{i}%
=(b_{i},b_{i+1}]$ for $i=1,...,n-1,$ $J_{n}=(b_{n},b_{n+1}),$ with
\[
b_{i}=a+(i-1){\frac{|I|}{n}}\quad\text{for }1\leq i\leq n+1.
\]
Let $f$ be the extremal function for $J^{a}(0,1),$ so that $f(0)=0$ and we may
suppose that $f$ is decreasing. We extend this function by oddness to the
interval $[-1,1]:$ then $f\chi_{(-1,0)}$ is the extremal function for
$J^{b}(-1,0)$). Also $f^{\prime}(-1)=f^{\prime}(1)=0$ and $f^{\prime}>0$ on
$(-1,1)$. Define $f_{i}$ to be the rescaling of $f$ from the interval $(-1,1)$
to $J_{i}$ with $\Vert f_{i}^{\prime\prime}\Vert_{p,J_{i}}=1$.

Now take $\alpha=(\alpha_{1},...,\alpha_{n-1})\in\mathbb{R}^{n-1}$ and
${\alpha_{n}}\in\mathbb{R}$ ($\alpha_{n}$ will depend on the selection of
$\alpha$). Let
\[
h(x):=\int_{a}^{x}\sum_{i=1}^{n}\alpha_{i}f_{i}^{\prime}(t)\chi_{J_{i}}(t)dt,
\]
selecting $\alpha_{n}$ in such a way that $h(b)=0$. Denote by $H$ the
collection of all such functions $h;$ plainly $\mathrm{rank}$ $H=n-1$. Then
with aid of Lemma \ref{Lemma 3.2} we have%

\begin{align*}
{\frac{\Vert h\Vert_{p,I}^{p}}{\Vert h^{\prime\prime}\Vert_{p,I}^{ p}}}  &
={\frac{\sum_{i=1}^{n}\Vert h\Vert_{p,J_{i}}^{p}}{\sum_{i=1}^{n}\Vert
h^{\prime\prime}\Vert_{p,J_{i}}^{p}}}={\frac{\sum_{i=1}^{n}\Vert\alpha_{i}%
\int_{b_{i}}^{.}f_{i}^{\prime}(t)dt+h(b_{i})\Vert_{p,J_{i}}^{p}}{\sum
_{i=1}^{n}\Vert\alpha_{i}f_{i}^{\prime\prime}\Vert_{p,J_{i}}^{p}}}\\
&  \geq{\frac{\sum_{i=1}^{n}\inf_{\lambda}\Vert\alpha_{i}\int_{b_{i}}^{.}%
f_{i}^{\prime}(t)dt-\lambda\Vert_{p,J_{i}}^{p}}{\sum_{i=1}^{n}\Vert\alpha
_{i}f_{i}^{\prime\prime}\Vert_{p,J_{i}}^{p}}}\\
&  ={\frac{\sum_{i=1}^{n}\inf_{\lambda}\Vert\alpha_{i}f_{i}-\lambda
\Vert_{p,J_{i}}^{p}}{\sum_{i=1}^{n}\Vert\alpha_{i}f_{i}^{\prime\prime}%
\Vert_{p,J_{i}}^{p}}}\\
&  ={\frac{\sum_{i=1}^{n}\Vert\alpha_{i}f_{i}\Vert_{p,J_{i}}^{p}}{\sum
_{i=1}^{n}\Vert\alpha_{i}f_{i}^{\prime\prime}\Vert_{p,J_{i}}^{p}}}={\frac
{\sum_{i=1}^{n}|\alpha_{i}|^{p}\Vert f_{i}\Vert_{p,J_{i}}^{p}}{\sum_{i=1}%
^{n}|\alpha_{i}|^{p}\Vert f_{i}^{\prime\prime}\Vert_{p,J_{i}}^{p}}}\\
&  ={\frac{\sum_{i=1}^{n}|\alpha_{i}|^{p}[J^{0}(J_{i})]^{p}\Vert f_{i}%
^{\prime\prime}\Vert_{p,J_{i}}^{p}}{\sum_{i=1}^{n}|\alpha_{i}|^{p}\Vert
f_{i}^{\prime\prime}\Vert_{p,J_{i}}^{p}}}\geq{\frac{|I|^{2p}}{n^{2p}}}\left(
J^{0}((0,1))\right)  ^{p}.
\end{align*}
Since $\mathcal{B}(0,1)=J^{0}(0,1),$ the proof is complete.
\end{proof}

We summarise these last two theorems in the following

\begin{theorem}
\label{Theorem 3.4}For the approximation numbers of the embedding $E$ we have
\[
\left(  {\frac{|I|}{n+1}}\right)  ^{2}\mathcal{B}(0,1)\leq a_{n}(E)\leq\left(
{\frac{|I|}{n-1}}\right)  ^{2}\mathcal{B}(0,1)
\]
and
\[
\lim_{n\rightarrow\infty}n^{2}a_{n}(E)=|I|^{2}\mathcal{B}(0,1).
\]
The same holds for all injective strict $s-$numbers of $E.$
\end{theorem}

Note that from results about eigenvalues for the uniform Euler-Bernoulli beam
problem when $p=2$ (see section 8.4 in \cite{Me}) it follows that the upper
and lower estimates for strict $s-$numbers of $E$ contained in Theorems
\ref{Theorem 3.1} and \ref{Theorem 3.3} cannot be essentially improved.

Let $\overset{a}{W}$\/$_{p}^{2}(I)$ be the closure in $W_{p}^{2}(I)$ of
$C_{a}^{\infty}(I)\cap W_{p}^{2}(I)$ (note that $C_{a}^{\infty}(I)$ is the
space of all functions in $C^{\infty}(I)$ that vanish at the left-hand
endpoint $a$ of the interval $I$). We equip $\overset{a}{W}$\/$_{p}^{2}(I)$
with the norm $\Vert u\Vert_{\overset{a}{{W}_{p}^{2}}(I)}:=\Vert u^{(2)}%
\Vert_{p,I}.$ We shall discuss the embedding $E_{a}:\overset{a}{W}$\/$_{p}%
^{2}(I)\rightarrow L_{p}(I),$ when $1<p<\infty$.

\begin{theorem}
\label{Theorem 3.5}For all $n\in\mathbb{N},$
\begin{equation}
a_{n+1}(E_{a})\leq{\frac{|I|^{2}}{(n-1/2)^{2}}}\mathcal{B}(0,1).
\label{Eq 3.2}%
\end{equation}
and
\begin{equation}
b_{n}(E_{a})\geq{\frac{|I|^{2}}{n^{2}}}\mathcal{B}(0,1). \label{Eq 3.3}%
\end{equation}

\end{theorem}

\begin{proof}
We start with the proof of (\ref{Eq 3.2}).

For each $n\in\mathbb{N}$ let $I=\cup_{i=0}^{n}I_{i}$, where $I_{i}%
=(a_{i},a_{i+1}]$ for $i=0,1,...,n-1,$ $I_{n}=(a_{n},a_{n+1}),$ with
\begin{align*}
a_{0}  &  =a,a_{1}=a+{\frac{|I|}{2n-1}},\\
a_{i}  &  =a+{\frac{|I|}{2n-1}}+(i-1){\frac{|I|}{n-1/2}}\quad\text{for
}1<i\leq n.
\end{align*}
Then we have $2|I_{0}|=|I_{i}|=(b-a)/(n-1/2)$ for $0<i\leq n$. For any
function $f$ on $I$ set
\[
Kf(x):=\chi_{I_{0}}(x)f(a)+\sum_{i=1}^{n-1}\chi_{I_{i}}(x)K_{I_{i}}f(x),
\]
where $K_{I}$ is given by (\ref{Eq 2.10}). By the same arguments as in the
proof of Theorem \ref{Theorem 3.1} we obtain (\ref{Eq 3.2}).

Now we prove (\ref{Eq 3.3}). As in the proof of Theorem \ref{Theorem 3.3} we
set: $I=\cup_{i=1}^{n}J_{i}$, where $J_{i}=(b_{i},b_{i+1}]$ for $i=1,...,n-1,$
$J_{n}=(b_{n},b_{n+1}),$ with
\[
b_{i}=a+(i-1){\frac{|I|}{n}}\quad\text{for }1\leq i\leq n+1,
\]
and consider the extremal functions $f_{i}$ with $\Vert f_{i}^{\prime\prime
}\Vert_{p,J_{i}}=1$.

Take $\alpha=(\alpha_{1},...,\alpha_{n})\in\mathbb{R}^{n}$ and let
\[
h(x):=\int_{a}^{x}\sum_{i=1}^{n}\alpha_{i}f_{i}^{\prime}(t)\chi_{J_{i}}(t)dt.
\]
Note that $h(a)=0$. Denote by $H$ the collection of all such functions $h;$
plainly $\mathrm{rank}$ $H=n$. Then as in the proof of Theorem
{\ref{Theorem 3.3}} we have%

\[
{\frac{\Vert h\Vert_{p,I}^{p}}{\Vert h^{\prime\prime}\Vert_{p,I}^{p}}}%
\geq{\frac{|I|^{2p}}{n^{2p}}}\left(  J^{0}((0,1))\right)  ^{p}={\frac
{|I|^{2p}}{n^{2p}}}\left(  \mathcal{B}((0,1))\right)  ^{p},
\]
which concludes the proof.
\end{proof}

As before we summarise the previous theorem in the following statement

\begin{theorem}
\label{Theorem 3.6}For the approximation numbers of the embedding $E_{a}$ we
have
\[
\left(  {\frac{|I|}{n}}\right)  ^{2}\mathcal{B}(0,1)\leq a_{n}(E_{a}%
)\leq\left(  {\frac{|I|}{n-3/2}}\right)  ^{2}\mathcal{B}(0,1)
\]
and
\[
\lim_{n\rightarrow\infty}n^{2}a_{n}(E_{a})=|I|^{2}\mathcal{B}(0,1).
\]
The same holds for all injective strict $s-$numbers of $E_{a}.$
\end{theorem}

\section{ Estimates for $s$-numbers of an operator of Volterra type}

Let us consider the Volterra type operator $T_{m}$ acting from $L_{p}(I)$ to
$L_{p}(I)$ and defined, for each $m\in\mathbb{N},$ by:%

\[
T_{m}f(x)=\int_{a}^{x}(x-t)^{m-1}f(t)dt,\qquad\text{when}\quad x\in I.
\]

We can see that
\[
{\frac{d^{m} }{dx^{m}}}\left(  T_{m} f\right)  (x)= f(x).
\]

For $T_{1}$ it is known that when $1<p<\infty$ we have:
\[
a_{n}(T_{1})=\gamma_{p}{\frac{|I|}{n-1/2}}.
\]
Here $a_{n}$ can be replaced by any strict s-number (see \cite{EL2}).

When $m>1$ no similar result is known, although if $p=2$ it is possible to
obtain estimates via the connection between approximation numbers and singular
numbers (see \cite{NS}).

We consider the case $m=2$ and note that $T_{2}:L_{p}(I)\rightarrow L_{p}(I)$
can be written as $T_{2}=E_{a}\circ S,$ where $S:L_{p}(I)\rightarrow
\overset{a}{W}$\/$_{p}^{2}(I)$ is given by%
\[
(Sf)(x)=\int\nolimits_{a}^{x}(x-t)f(t)dt
\]
and $\left\Vert S\right\Vert =1.$ This isometric relationship with
$E_{a}:\overset{a}{W}$\/$_{p}^{2}(I)\rightarrow L_{p}(I)$ means that we can,
instantly, state the following theorems:

\begin{theorem}
\label{Theorem 4.1}For all $n\in\mathbb{N},$
\begin{equation}
a_{n+1}(T_{2})\leq{\frac{|I|^{2}}{(n-1/2)^{2}}}\mathcal{B}(0,1).
\label{Eq 4.2}%
\end{equation}
and
\begin{equation}
b_{n}(T_{2})\geq{\frac{|I|^{2}}{n^{2}}}\mathcal{B}(0,1). \label{Eq 4.3}%
\end{equation}

\end{theorem}

\begin{theorem}
\label{Theorem 4.2}For the approximation numbers of the operator $T_{2}$ we
have
\[
\left(  {\frac{|I|}{n}}\right)  ^{2}\mathcal{B}(0,1)\leq a_{n}(T_{2}%
)\leq\left(  {\frac{|I|}{n-3/2}}\right)  ^{2}\mathcal{B}(0,1)
\]
and
\[
\lim_{n\rightarrow\infty}n^{2}a_{n}(T_{2})=|I|^{2}\mathcal{B}(0,1).
\]
The same holds for all injective strict $s-$numbers of $T_{2}.$
\end{theorem}

\begin{acknowledgement}
The authors would like to thank Prof. B. Kawohl for information about the
literature and Prof. P. Dr\'{a}bek and Prof. J. Benedikt for their valuable comments.
\end{acknowledgement}

\bigskip D. E. Edmunds, Department of Mathematics, University of Sussex,
Pevensey 2 Building, Falmer, Brighton BN1 9QH, UK

\ \ \ \ email address: davideedmunds@aol.com

\ \ \  telephone +44(0)1273558392; fax +44(0)1273876677 

\medskip J. Lang, Department of Mathematics, Ohio State University, 100 Math
Tower, 231 West 18th Avenue, Columbus, OH 43210-1174, USA

\medskip\ \ \ \ \ email address: lang@math.ohio-state.edu

\ \ \  telephone (614)292-9133; fax (614)292-5817  


\end{document}